\documentclass[10pt,a4paper]{article}
\pdfoutput=1

\usepackage{geometry}
\usepackage{graphicx}
\usepackage{amsmath,amssymb,amsthm,mathtools}
\usepackage{tikz}
\usepackage{hyperref}
\usepackage[siunitx]{circuitikz}
\usepackage{suffix}
\usetikzlibrary{shapes.geometric}
\usetikzlibrary{arrows,chains,scopes}

\title{Structure-preserving discretization for port-Hamiltonian descriptor systems}
\author{
  Volker Mehrmann\footnote{
    Institut f\"ur Mathematik MA 4-5, TU Berlin, Str.~des 17.~Juni 136, D-10623 Berlin, FRG.
    \textit{Email address: }\texttt{mehrmann@math.tu-berlin.de}.}\; and
  Riccardo Morandin\footnote{
    Institut f\"ur Mathematik MA 4-5, TU Berlin, Str.~des 17. Juni 136, D-10623 Berlin, FRG.
    \textit{Email address: }\texttt{morandin@math.tu-berlin.de}.}
}

\newcommand{\mB}[1]{\mathbb{#1}}
\newcommand{\mc}[1]{\mathcal{#1}}
\newcommand{\R}{\mB{R}}

\newcommand{\N}{\mB{N}}
\newcommand{\mat}[1]{\begin{matrix}#1\end{matrix}}
\newcommand{\bmat}[1]{\begin{bmatrix}#1\end{bmatrix}}

\newcommand{\pa}{\partial}
\newcommand{\dd}[2]{\frac{\td#1}{\td#2}}
\newcommand{\pd}[2]{\frac{\pa#1}{\pa#2}}
\newcommand{\dualp}[2]{\langle #1 \,|\, #2 \rangle}

\newcommand{\td}{\textup{d}}
\newcommand{\set}[1]{\left\{#1\right\}}
\newcommand{\pset}[1]{\left(#1\right)}
\newcommand{\bset}[1]{\left[#1\right]}
\newcommand{\bilp}[1]{\llangle #1 \rrangle}

\newcommand{\fc}[2]{\mc C^{#1}(\mc S,#2)}
\WithSuffix\newcommand\fc*[1]{\mc C(\mc S,#1)}

\newcommand\independent{\protect\mathpalette{\protect\independenT}{\perp}}
\def\independenT#1#2{\mathrel{\rlap{$#1#2$}\mkern2mu{#1#2}}}

\makeatletter
\newsavebox{\@brx}
\newcommand{\llangle}[1][]{\savebox{\@brx}{\(\m@th{#1\langle}\)}%
  \mathopen{\copy\@brx\kern-0.5\wd\@brx\usebox{\@brx}}}
\newcommand{\rrangle}[1][]{\savebox{\@brx}{\(\m@th{#1\rangle}\)}%
  \mathclose{\copy\@brx\kern-0.5\wd\@brx\usebox{\@brx}}}
\makeatother

\newtheorem{theorem}{Theorem}
\newtheorem{lemma}{Lemma}
\theoremstyle{definition}
\newtheorem{definition}{Definition}
\theoremstyle{remark}
\newtheorem{remark}{Remark}

\DeclareMathOperator{\diag}{diag}

\begin{document}

\maketitle

\begin{abstract}

We extend the modeling framework of port-Hamiltonian descriptor systems to include under-~and over-determined systems and arbitrary differentiable Hamiltonian functions.
This structure is associated with a Dirac structure that encloses its energy balance properties.
In particular, port-Hamiltonian systems are naturally passive and Lyapunov stable, because the Hamiltonian defines a Lyapunov function.
The explicit representation of input and dissipation in the structure make these systems particularly suitable for output feedback control.
It is shown that this structure is invariant under a wide class of nonlinear transformations, and that it can be naturally modularized, making it adequate for automated modeling.
We investigate then the application of time-discretization schemes to these systems and we show that, under certain assumptions on the Hamiltonian, structure preservation is achieved for some methods.
Relevant examples are provided.

\end{abstract}

\section{Introduction}

The port-Hamiltonian (pH) structure \cite{SchJ14} is very appealing for modeling, simulation and control of complex multiphysics systems.
PH systems generalize classical Hamiltonian systems and gradient systems by including energy dissipation and interaction with the environment (represented by \emph{ports}).
Implicit formulations for pH systems enclosing their properties are given through objects from differential geometry, known as Dirac structures.
Similarly as pure Hamiltonian systems, pH systems also gain from the application of geometric integration methods \cite{HaiLW06}, like the Gauss-Legendre collocation schemes. With additional requirements imposed on the form of the Hamiltonian function, structure preseration can be achieved, by a proper discretization of the Dirac structure \cite{KotL18}.

The energy concept can be used as a common language, to interconnect different pH systems while mantaining the structure, even when they originate in different physical domains, that may include mechanical, mechatronic, fluidic, thermic, hydraulic, pneumatic, elastic, plastic or electric components.
The explicit incorporation of constraints, often inavoidable when using modeling packages such as \textsc{Modelica} (\verb|https://www.modelica.org|), \textsc{Matlab/Simulink} (\verb|https://www.mathworks.org|) or \textsc{Simpack} (\verb|http://www.simpack.com|), produce differential-algebraic equations (DAEs), also referred to as \emph{descriptor systems}, which may contain hidden constraints, consistency requirements for initial conditions and additional regularity requirements.
A definition for linear time-varying pH descriptor systems (pHDAEs) has been given in \cite{BeaMXZ18}.
While including many interesting examples, that description falls short of fully extending conventional pH systems, since it is restricted to quadratic Hamiltonian functions and finite dimensional states, where the dimension is the same as the number of equations, and the matrix coefficients are independent from state.

In this paper we extend the definition of pHDAEs to include arbitrary differentiable Hamiltonian functions and systems with a different number of states and equations. The new definition extends to the case of infinite-dimensional states and time- and space-varying coefficients.
We also give a new definition of Dirac structure, so that we can always associate one to our pHDAEs.
We generalize the results in \cite{KotL18} to the case of pHDAEs, so that we can apply structure-preserving time discretization to port-Hamiltonian descriptor systems.

The paper is organized as follows.
In Section II, we introduce our new definition of pHDAEs, and we investigate its properties.
In Section III, we study the application of collocation schemes to pHDAEs and conditions for which structure preservation is achieved.
In Section IV, we illustrate a simple example of a pHDAE from the electrical circuit domain, we exploit the pH structure to apply control to the system and we present numerical experiments using Gauss-Legendre collocation.

\section{Port-Hamiltonian descriptor systems}

\subsection{Formulation}

\begin{definition}[pHDAE]
  Let us consider a time interval $\mB I\subseteq\R$, a state space $\mc X\subseteq\R^n$ and the extended state space $\mc S=\mB I\times\mc X$.
  A \emph{port-Hamiltonian descriptor system} (or \emph{pHDAE}) is a system of differential (-algebraic) equations of the form
  \begin{equation}\label{eq:pHDAE}
    \begin{split}
      E(t,x)\dot x + r(t,x) &= (J(t,x)-R(t,x))z(t,x) + (B(t,x)-P(t,x))u, \\
      y &= (B(t,x)+P(t,x))^Tz(t,x) + (S(t,x)-N(t,x))u,
    \end{split}
  \end{equation}
  associated with an \emph{Hamiltonian function} $\mc H\in\fc1\R$,
  where $x(t)\in\mc X$ is the state, $u(t),y(t)\in\R^m$ are the input and output,
  $E\in\fc*{\R^{\ell,n}}$ is the \emph{flow matrix},
  $r,z\in\fc*{\R^\ell}$ are the \emph{time-flow} and \emph{effort} functions,
  $J,R\in\fc*{\R^{\ell,\ell}}$ are the \emph{structure} and \emph{dissipation} matrices,
  $B,P\in\fc*{\R^{\ell,m}}$ are the \emph{port matrices} and
  $S,N\in\fc*{\R^{m,m}}$ are the \emph{feed-through matrices}.
  Furthermore, the following properties must hold:
  \begin{enumerate}
    \item The extended structure and dissipation matrices $\Gamma,W\in\fc*{\R^{\ell+m,\ell+m}}$, defined as
    \begin{equation}
      \Gamma := \bmat{J & B \\ -B^T & N}, \qquad W := \bmat{R & P \\ P^T & S}
    \end{equation}
    satisfy $\Gamma=-\Gamma^T$ and $W=W^T\geq0$ pointwise.
    \item The gradient of the Hamiltonian satisfies $\pa_x\mc H=E^Tz$ and $\pa_t\mc H=z^Tr$ pointwise.
  \end{enumerate}
  %
  This definition can be extended to the case of weak solutions and state spaces with infinite dimension.
  In particular, this framework can be also used to describe partial-differential-algebraic equations (PDAEs).
  For simplicity, the results and examples presented in this paper will be limited to the finite-dimensional case.
\end{definition}

\begin{remark}
  The definition of pH system often includes extra conditions on the Hamiltonian, e.g.~that it is a convex function, or that it is always non-negative \cite{BeaMXZ18}. While these conditions are not necessary in general, they are often satisfied, and they strengthen the stability properties derived from the pH structure.
\end{remark}

\subsection{Properties}

Port-Hamiltonian descriptor systems as defined in \eqref{eq:pHDAE} satisfy many useful properties.
We list and prove some of these in the following.

\subsubsection{Dissipation inequality} Any port-Hamiltonian system must satisfy some kind of dissipation inequality, expressing its passivity, i.e.~energy cannot be created within the system.

\begin{theorem}[dissipation inequality]
  Let us consider a pHDAE of the form \eqref{eq:pHDAE}. Then the power balance equation (PBE)
  \begin{equation}\label{eq:powerBalanceEq}
    \dd{}{t}\mc H(t,x(t)) = - \bmat{z \\ u}^TW\bmat{z \\ u} + u^Ty
  \end{equation}
  holds along any solution $x$, for any input $u$.
  In particular, the \emph{dissipation inequality}
  \begin{equation}\label{eq:dissIneq}
    \mc H(t_2,x(t_2)) - \mc H(t_1,x(t_1)) \leq \int_{t_1}^{t_2}u(\tau)^Ty(\tau)\td\tau
  \end{equation}
  holds.
\end{theorem}
\begin{proof}
  The pHDAE \eqref{eq:pHDAE} can be written as
  \begin{equation*}
    \bmat{E\dot x+r \\ 0} = (\Gamma-W)\bmat{z \\ u} + \bmat{0 \\ y},
  \end{equation*}
  in particular
  \begin{align*}
    \dd{}{t}\mc H(t,x(t)) &= \pa_t\mc H + \pa_x\mc H^T\dot x = z^T(E\dot x+r)
    = \bmat{z \\ u}^T\bmat{E\dot x+r \\ 0} = \\
    &= \bmat{z \\ u}^T\pset{(\Gamma-W)\bmat{z \\ u} + \bmat{0 \\ y}}
    = -\bmat{z \\ u}^TW\bmat{z \\ u} + u^Ty,
  \end{align*}
  which is \eqref{eq:powerBalanceEq}. By time integration, since $W=W^T\geq0$, we immediately get \eqref{eq:dissIneq}.
\end{proof}

Note that, with no input ($u=0$), if the Hamiltonian $\mc H$ is locally positive-definite in an equilibrium point $x^*$ (up to shifting it by a constant), then $\mc H$ is a Lyapunov function and the system is stable.

\subsubsection{Variable transformations}

This class of pHDAEs is closed under many variable transformations:
\begin{theorem}\label{thm:varTrans}
  Consider a pHDAE of the form \eqref{eq:pHDAE}. Let $\tilde{\mc X}\subseteq\R^{\tilde n}$ be a second state space, let $\tilde{\mc S}:=\mB I\times\tilde{\mc X}$, let $x=\varphi(t,\tilde x)\in\mc C^1(\tilde{\mc S},\mc X)$ be a local diffeomorphism (with respect to $\tilde x$) and let $U\in\mc C(\tilde{\mc S},\R^{\ell,\ell})$ be pointwise invertible.
  Consider the input-output DAE
  \begin{equation}\label{eq:tformPHDAE}
  \begin{split}
    \tilde E\dot{\tilde x} + \tilde r &= (\tilde J-\tilde R)\tilde z + (\tilde B-\tilde P)u, \\
    y &= (\tilde B+\tilde P)^T\tilde z + (S-N)u,
  \end{split}
  \end{equation}
  with $\tilde E=U^T(E\circ\varphi)\pa_{\tilde x}\varphi$, $\tilde J=U^T(J\circ\varphi)U$, $\tilde R=U^T(R\circ\varphi)U$, $\tilde B=U^T(B\circ\varphi)$, $\tilde P=U^T(P\circ\varphi)$, $\tilde z=U^{-1}(z\circ\varphi)$ and $\tilde r=U^T(r\circ\varphi+(E\circ\varphi)\pa_t\varphi)$,
  where we denote $(F\circ\varphi)(t,\tilde x)=F(t,\varphi(t,\tilde x))$ for any $F\in\mc C(\mc S,\cdot)$, and let $\tilde{\mc H}(t,\tilde x):=(\mc H\circ\varphi)(t,\tilde x)$.
  Then \eqref{eq:tformPHDAE} is a pHDAE with Hamiltonian function $\tilde{\mc H}$, and to any solution $(\tilde x,u,y)$ of \eqref{eq:tformPHDAE} corresponds a solution $(x,u,y)$ of \eqref{eq:pHDAE} with $x(t)=\varphi(t,\tilde x(t))$.
  Furthermore, if $\varphi(t,\cdot)$ is a global diffeomorphism for all $t\in\mB I$, then the two systems are equivalent.
\end{theorem}
\begin{proof}
  The transformed DAE system is obtained from \eqref{eq:pHDAE} by setting $x=\varphi(t,\tilde x)$, pre-multiplying with $U^T$ and inserting $UU^{-1}$ in front of $z$ in the first equation.
  It is then clear that if $(\tilde x,u,y)$ is a solution of \eqref{eq:tformPHDAE}, then $(x,u,y)$ is a solution of the original system.
  If $\varphi(t,\cdot)$ is a global diffeomorphism, then we can apply the inverse tranformation $(\varphi^{-1},U^{-1})$, where $\varphi^{-1}$ is to intended as $\varphi(t,\varphi^{-1}(t,x))=x$, and get a solution $(\tilde x,u,y)$ for any solution $(x,u,y)$ of the original system, making the two systems equivalent.
  
  To show that \eqref{eq:tformPHDAE} is still a pHDAE, we must check that conditions 1 and 2 in the definition are satisfied.
  \begin{enumerate}
    \item By substitution, we get
    \begin{equation*}
    \begin{split}
      \tilde W &= \bmat{\tilde R & \tilde P \\ \tilde P^T & S}
      = \bmat{U^T(R\circ\varphi)U & U^T(P\circ\varphi) \\ (P\circ\varphi)^TU & S\circ\varphi} = \\
      &= \bmat{U & 0 \\ 0 & I}^T (W\circ\varphi) \bmat{U & 0 \\ 0 & I} \geq 0.
    \end{split}
    \end{equation*}
    \item By substitution, we get
    \begin{equation*}
    \begin{split}
      \pd{\tilde{\mc H}}{\tilde x}(t,\tilde x) &= \pd{\varphi}{\tilde x}^T\pset{\pd{\mc H}{x}\circ\varphi}
      = \pd{\varphi}{\tilde x}^T\pset{E^Tz\circ\varphi} = \\
      &= \pd{\varphi}{\tilde x}^T(E^T\circ\varphi)UU^{-1}(z\circ\varphi)
      = \tilde E^T\tilde z
    \end{split}
    \end{equation*}
    and
    \begin{equation*}
    \begin{split}
      \pd{\tilde{\mc H}}{t}(t,\tilde x) &=
      \pd{\mc H}{t}\circ\varphi + \pset{\pd{\mc H}{x}\circ\varphi}^T\pd{\varphi}{t} = \\
      &= z^Tr\circ\varphi + \pset{z^TE\circ\varphi}\pd{\varphi}{t} = \\
      &= (z\circ\varphi)^T\pset{r\circ\varphi + (E\circ\varphi)\pd{\varphi}{t}} = \\
      &= \tilde z^TU^TU^{-T}\tilde r = \tilde z^T\tilde r.
    \end{split}
    \end{equation*}
  \end{enumerate}
  This concludes the proof.
\end{proof}

\subsubsection{Autonomous form} Any DAE can be made autonomous by adding time as a state. In particular, any pHDAE can be made autonomous without destroying the structure:
\begin{equation}\label{eq:pHDAEauto}
  \begin{split}
    \bmat{E & r \\ 0 & 1}\bmat{\dot x \\ \dot t} &=
    \bmat{J-R & 0 \\ 0 & 0}\bmat{z \\ 0} + \bmat{B-P & 0 \\ 0 & 1}\bmat{u \\ 1}, \\
    \bmat{y \\ 0} &= \bmat{B+P & 0 \\ 0 & 1}^T\bmat{z \\ 0} + \bmat{S-N & 0\\0 & 0}\bmat{u \\ 1}.
  \end{split}
\end{equation}%
Note that condition 2 becomes simply $\nabla_{\tilde x}\tilde{\mc H}=\tilde E^T\tilde z$, where the tilde denotes the quantities in the autonomous system.

\subsubsection{Structure-preserving interconnection}

Let us consider two autonomous pHDAEs of the form
\begin{align*}
  E_i\dot x_i &= (J_i-R_i)z_i + (B_i-P_i)u_i, \\
  y_i &= (B_i+P_i)^Tz_i + (S_i-N_i)u_i,
\end{align*}
with Hamiltonian $\mc H_i$, for $i=1,2$, and assume that the aggregated input $u=(u_1,u_2)$ and output $y=(y_1,y_2)$ satisfy a linear interconnection relation $Mu+Ny=0$ for some $M,N\in\fc*{\R^{k,m}}$. Then the aggregated system can be written as a pHDAE of the form
\begin{align*}
  \bmat{E & 0 & 0 \\ 0 & 0 & 0 \\ 0 & 0 & 0 \\ 0 & 0 & 0}\bmat{\dot x \\ \dot{\hat u} \\ \dot{\hat y}} &=
  \bmat{\Gamma-W & \mat{0 & 0 \\ I_m & -M^T}\hskip-5pt \\ \mat{0 & -I_m \\ 0 & M}\hskip-5pt & \mat{0 & -N^T \\ N & 0}\hskip-5pt}
  \bmat{z \\ \hat u \\ \hat y \\ 0}
  + \bmat{0 \\ 0 \\ I_m \\ 0}u, \\
  y &= \hat y,
\end{align*}%
with Hamiltonian $\mc H=\mc H_1+\mc H_2$, where we have introduced new state variables $\hat u,\hat y\in\R^n$ that copy $u,y$, and we aggregated $E=\diag(E_1,E_2)$, $x=(x_1,x_2)$, $z=(z_1,z_2)$, $m=m_1+m_2$, $\Gamma=\Pi\diag(\Gamma_1,\Gamma_2)\Pi^T$ and $W=\Pi\diag(W_1,W_2)\Pi^T$, where $\Pi\in\R^{\ell+m,\ell+m}$ is the permutation matrix
\begin{equation*}
  \Pi = \bmat{I_{\ell_1} & 0 & 0 & 0 \\ 0 & 0 & I_{\ell_2} & 0 \\ 0 & I_{n_1} & 0 & 0 \\ 0 & 0 & 0 & I_{n_2}}.
\end{equation*}
%

Note that, in general, we may not be able to reduce the number of inputs and outputs.
If we also assume that the interconnection is energy-preserving (e.g.~if $Mu+Ny=0$ defines a Dirac structure for $(y,u)$), then index reduction \cite{KunM06} and row operations can usually be applied to make the system smaller.

\subsection{Dirac structure}

Port-Hamiltonian systems are ofter described through differential geometric structures known as Dirac structures \cite{SchJ14}. We present first the basic definitions.
\begin{definition}[linear Dirac structure]
  Let $\mc F$ be a linear space and $\mc E:=\mc F^*$ its dual space. Let $\bilp{\cdot,\cdot}$ be the bilinear form on $\mc F\times\mc E$ defined as
  \begin{equation*}
    \bilp{(f_1,e_1),(f_2,e_2)} := \dualp{e_1}{f_2} + \dualp{e_2}{f_1},
  \end{equation*}
  where $\dualp{\cdot}{\cdot}$ denotes the duality pairing.
  A \emph{Dirac structure} on $\mc F\times\mc E$ is then a linear subspace $\mc D\subseteq\mc F\times\mc E$, such that $\mc D=\mc D^{\independent}$. 
\end{definition}
In particular, in finite dimension, one only needs to prove $\dim\mc D=\dim\mc F$ and $\dualp{e}{f}=0$ for all $(f,e)\in\mc D$.
If $(f,e)\in\mc D$, then $f$ and $e$ are called \emph{flow} and \emph{effort}, respectively.
In \cite{SchJ14}, the more general definition of modulated Dirac structure over $\mc X$ is also presented, as a subbundle of $T\mc X\oplus T^*\mc X$, that denotes the Whitney sum between the tangent and cotangent bundles of $\mc X$.
%
%
To introduce a Dirac structure for pHDAEs, we need to extend this further.
\begin{definition}
  Consider a state space $\mc X$ and a vector bundle $\mc V$ over $\mc X$ with fibers $\mc V_x$.
  A \emph{Dirac structure} over $\mc V$ is a subbundle $\mc D \subseteq \mc V \oplus \mc V^*$ such that, for all $x\in\mc X$, $\mc D_x\subseteq\mc V_x\times\mc V_x^*$ is a linear Dirac structure.
\end{definition}
Note that modulated Dirac structures are a special case of our definition, where $\mc V=T\mc X$.
To associate a Dirac structure to our pHDAE system, we first prove the following lemma:
\begin{lemma}\label{lem:dirac}
  Let $\mc D\subseteq\mc V\oplus\mc V^*$ be a vector subbundle with fibers defined by
  \begin{equation*}
    \mc D_x=\set{(f,e)\in\mc V_x\times\mc V_x^*:f+J(x)e=0},
  \end{equation*}
  where $J:\mc X\to\mc L(\mc V_x^*,\mc V_x)$ is a skew-symmetric operator.
  Then $\mc D$ is a Dirac structure.
\end{lemma}
\begin{proof}
  For generic $(f,e)\in\mc D_x$ and $(f',e')\in\mc V_x\times\mc V_x^*$,
  \begin{align*}
    \bilp{(f,e),(f',e')} &= \dualp{e}{f'} + \dualp{e'}{f} = \\ 
    &= \dualp{e}{f'}-\dualp{e'}{J e} = \dualp{e}{f'+Je'}.
  \end{align*}
  We show that $(f',e')\in\mc D_x$ if and only if $\dualp{e}{f'+Je'}=0$ for all $(f,e)\in\mc D_x$.
  On one hand, if $(f',e')\in\mc D_x$, then $\dualp{e}{f'+Je'}=0$ holds for any $e\in\mc E$.
  On the other hand, if $(f',e')\notin\mc D_x$, then $f'+Je'\neq0$ and so $\exists e\in\mc E$ such that $\dualp{e}{f'+Je'}=1$, but $(f,e)\in\mc D_x$ with $f=-Je$.
\end{proof}
We can now associate a Dirac structure to our pHDAE system in the following way:
\begin{theorem}\label{thm:diracStructure}
  Given an autonomous pHDAE, let us define the \emph{flow fiber} $\mc V_x=\mc F_x^s\times\mc F_x^p\times\mc F_x^d$ for all $x\in\mc X$, where $\mc F_x^s:= E(x)T_x\mc X\subseteq\R^\ell$ is the storage flow fiber,
  $\mc F_x^p:=\R^m$ is the port flow fiber and $\mc F_x^d:=\R^{\ell+m}$ is the dissipation flow fiber.
  Let us partition $f=(f_s,f_p,f_d)\in\mc V$ and $e=(e_s,e_p,e_d)\in\mc V^*$.
  Then the subbundle $\mc D\subseteq\mc V\oplus\mc V^*$ with
  {
  \begin{align*}
    \mc D_x = \Bigg\{(f,e)\in\mc V_x\times\mc V_x^* \;\Bigg|\;
    f + \bmat{\Gamma(x) & I_{\ell+m} \\ -I_{\ell+m} & 0} e = 0 \Bigg\}    
  \end{align*}}%
  is a Dirac structure over $\mc V$.
  Furthermore, the system of equations
  \begin{equation}\label{eq:feCond}
    \begin{alignedat}{2}
      f_s &= -E(x)\dot x, &\qquad e_s &= z(x), \\
      f_p &= y, &\qquad e_p &= u, \\
      e_d &= -W(x)f_d, &\qquad (f,e)&\in\mc D_x
    \end{alignedat}
  \end{equation}
  is equivalent to the original pHDAE, and $\dualp{e}{f}=0$ represents the power balance equation.
\end{theorem}
\begin{proof}
  $\mc D$ is a Dirac structure because of Lemma \ref{lem:dirac}.
  Note that the pHDAE can be written in compact form as
  \begin{equation*}
    \bmat{E(x)\dot x \\ -y} = (\Gamma(x)-W(x))\bmat{z(x) \\ u}.
  \end{equation*}
  The condition $(f,e)\in\mc D_x$ can be written as
  \begin{equation*}
    \bmat{-f_s \\ -f_p} = \Gamma(x)e_s + e_p, \qquad
    f_d = (e_s,e_p),
  \end{equation*}
  that together with the conditions \eqref{eq:feCond} is equivalent to
  \begin{equation*}
    \begin{split}
      \bmat{E(x)\dot x \\ -y} &= (\Gamma(x)-W(x))\bmat{z(x) \\ u}, \\
      f &= \pset{-E(x)\dot x,\; y,\; \bmat{z(x) \\ u}}, \\
      e &= \pset{z(x),\; u,\; -W(x)\bmat{z(x) \\ u}},
    \end{split}
  \end{equation*}
  which is exactly the compact form of the pHDAE, plus the definition for the flow and effort.
  Finally, note that the equation $\dualp{e}{f}=0$ can be written as
  \begin{equation*}
    \begin{split}
      0 &= \dualp{z(x)}{-E(x)\dot x} + \dualp{u}{y} + \dualp{-W(x)\bmat{z(x)\\u}}{\bmat{z(x)\\u}} = \\
      &= \dd{}{t}\mc H(x) + y^Tu - \bmat{z(x) \\ u}^TW(x)\bmat{z(x) \\ u},
    \end{split}
  \end{equation*}
  which is the power balance equation.
\end{proof}


Note that, if we want to retrieve a pHDAE system from a Dirac structure, the additional conditions \eqref{eq:feCond} and the definition of $\mc H(x)$ are needed.
These can also be lifted to a geometric interpretation, by the means of a Lagrangian submanifold and of a dissipative structure \cite{SchM18}.

\section{Time discretization}

Let us consider a finite-dimensional pHDAE of the form \eqref{eq:pHDAE}.
Under some regularity assumptions \cite{KunM06}, many classes of Runge-Kutta methods can be applied to compute time-discretizations of differential-algebraic equations.
An important class of Runge-Kutta methods is the class of collocation methods.
We extend the results from \cite{KotL18} to pHDAEs, proceeding in a similar way.

\subsection{Collocation methods for DAEs}

Assume that an input function $u:\mB I\times\mc X\to\R^m$, depending on time and possibly space, is given.
Given a time interval $\mB I=[t_0,t_f]$ of length $h=t_f-t_0$, and a consistent initial condition $x_0$ at time $t_0$, we approximate the solution $x(t)$ of the pHDAE on $\mB I$ with a polynomial $\tilde x(t)\in\R[t]_s$, where $\R[t]_s$ denotes the vector space of polynomials of degree at most $s$.
The polynomial is chosen such that $\tilde x(t_0)=x_0$, and that it satisfies the differential-algebraic equation in $s$ collocation points $t_i=t_0+h\gamma_i$ with $\gamma_i\in[0,1]$ for $i=1,\ldots,s$.

Let $\ell_i$ denote the $i$-th Lagrange interpolation polynomial with respect to the nodes $\gamma_1,\ldots,\gamma_s$, i.e.,
\begin{equation*}
  \ell_i(\tau) := \prod_{\substack{j=1\\j\neq i}}^s\frac{\tau-\gamma_j}{\gamma_i-\gamma_j}.
\end{equation*}
Then we can write
\begin{equation*}
  \begin{split}
    \dot{\tilde x}(t_0+\tau h) &= \sum_{i=1}^s\dot x_i\ell_i(\tau), \\
    \tilde x(t_0+\tau h) &= x_0 + h\sum_{j=1}^s\dot x_j\int_0^\tau\ell_j(\sigma)\td\sigma,
  \end{split}
\end{equation*}
for certain $\dot x_i=\dot{\tilde x}(t_i)$ that we have to compute, and also
\begin{align*}
  x_i &:= \tilde x(t_i) = x_0 + h\sum_{j=1}^s\alpha_{ij}\dot x_j, \\
  x_f &:= \tilde x(t_f) = x_0 + h\sum_{j=1}^s\beta_j\dot x_j,
\end{align*}
where $\alpha_{ij}:=\int_0^{\gamma_i}\ell_j(\sigma)\td\sigma$ and $\beta_j:=\int_0^1\ell_j(\sigma)\td\sigma$.
In particular, the constants $\alpha_{ij},\beta_j,\gamma_i\in\R$ for $i,j=1\ldots s$ are the coefficients of the Butcher diagram of the associated Runge-Kutta method \cite{HaiLW06}.


\subsection{Dirac structure associated to discretization}

Consider for simplicity the autonomous case.
The non-autonomous case is similar, with a more cumbersome notation.
Let $\mc D_x$ be the Dirac structure associated to \eqref{eq:pHDAEauto} (as in Theorem \ref{thm:diracStructure}).
We define the Dirac structure associated to the time-discretization as $\set{\mc D_{x_i}:i=1,\ldots,s}$, i.e.,
\begin{align*}
  \mc D_{x_i} &=
  \Bigg\{(f^i,e^i)\in\mc V_{x_i}\times\mc V_{x_i}^* \;\Bigg|\;
  f^i +
    \bmat{
        \Gamma(t_i,x_i) & I_{\ell+m} \\
        -I_{\ell+m} & 0
    }
    e^i
    = 0
  \Bigg\},
\end{align*}
with $f^i=(f_s^i,y_i,f_d^i)$ and $e^i=(e_s^i,u_i,e_d^i)$.
Taking $(f^i,e^i)\in\mc D_{x_i}$, together with
\begin{subequations}
\begin{align}
  x_f &= x_0 + h\sum_{i=1}^s\beta_i\dot x_i,    \label{eq:coll1} \\
  x_i &= x_0 + h\sum_{j=1}^s\alpha_{ij}\dot x_j,  \label{eq:coll2} \\
  f_s^i &= -E(x_i)\dot x_i,             \label{eq:coll3} \\
  e_s^i &= z(x_i),          \label{eq:coll4} \\
  e_d^i &= -W(x_i)f_d^i,        \label{eq:coll5} \\
  u_i &= u(x_i),          \label{eq:coll6} 
\end{align}
\end{subequations}
we get a system that is equivalent to applying the collocation method and computing the \emph{discrete input} and \emph{output} $u_i,y_i$, for $i=1,\ldots,s$.
Let us define the \emph{collocation flows, efforts, input and output} in $\R[t]_{s-1}$ as
\begin{alignat*}{2}
  \tilde{f}_s(t_0+h\tau) &= \sum_{i=1}^s f_s^i\ell_i(\tau), \quad&
  \tilde{e}_s(t_0+h\tau) &= \sum_{i=1}^s e_s^i\ell_i(\tau), \\
  \tilde{f}_d(t_0+h\tau) &= \sum_{i=1}^s f_d^i\ell_i(\tau), \quad&
  \tilde{e}_d(t_0+h\tau) &= \sum_{i=1}^s e_d^i\ell_i(\tau), \\
  \tilde y(t_0+h\tau) &= \sum_{i=1}^s y_i\ell_i(\tau), \quad&
  \tilde u(t_0+h\tau) &= \sum_{i=1}^s u_i\ell_i(\tau).
\end{alignat*}
Note that, by construction, $(\tilde{f}_s,\tilde y,\tilde{f}_d;\tilde{e}_s,\tilde u,\tilde{e}_d)\in\mc D_{\tilde x}$ in all collocation points $t_i$.
Let us consider the evolution of the Hamiltonian $\mc H$ along the collocation polynomial $\tilde x(t)$.
In particular, let $H(t):=\mc H(\tilde x(t))$: we can then write
\begin{equation*}
  H(t)-H(t_0) = \int_{t_0}^t\dot{H}(s)\td s.
\end{equation*}
In the collocation points, the PBE is satisfied:
\begin{equation*}
  \begin{split}
    \dot{H}(t_i) &= \nabla\mc H(x_i)^T\dot x_i = z(x_i)^TE(x_i)\dot x_i = \\
    &= -\dualp{e_s^i}{f_s^i} = \dualp{e_d^i}{f_d^i} + \dualp{y_i}{u_i},
  \end{split}
\end{equation*}
for $i=1,\ldots,s$.
In particular, if we apply the quadrature rule associated to the collocation method, we get
\begin{equation*}
  \begin{split}
    & H(t_f) - H(t_0) = h\sum_{j=1}^s\beta_j\dot{H}(t_j) + \mc O(h^{p+1}) = \\
    &\qquad= -h\sum_{j=1}^s\beta_j\dualp{e_s^j}{f_s^j} + \mc O(h^{p+1}) = \\
    &\qquad= h\sum_{j=1}^s\beta_j\dualp{e_d^j}{f_d^j} + h\sum_{j=1}^s\beta_j\dualp{y_j}{u_j} + \mc O(h^{p+1}),
  \end{split}
\end{equation*}
where $p\in\N$ is the degree of exactness of the quadrature rule.
%
We observe that, for the same reason,
\begin{subequations}
  \begin{align}
    \label{eq:discDissInt}
    h\sum_{j=1}^s\beta_j\dualp{e_d^j}{f_d^j} &= \int_{t_0}^{t_f}\dualp{\tilde{e}_d}{\tilde{f}_d} + \mc O(h^{p+1}), \\
    \label{eq:discPortInt}
    h\sum_{j=1}^s\beta_j\dualp{y_j}{u_j} &= \int_{t_0}^{t_f}\dualp{\tilde y}{\tilde u} + \mc O(h^{p+1}),
  \end{align}
\end{subequations}
so
\begin{equation*}
  H(t_f) - H(t_0) = \int_{t_0}^{t_f}\Big(\dualp{\tilde{e}_d(s)}{\tilde{f}_d(s)}
  + \dualp{\tilde y(s)}{\tilde u(s)}\Big)\td s
  + \mc O(h^{p+1}).
\end{equation*}
We note that, if $p\geq 2s-2$, then equations \eqref{eq:discDissInt} and \eqref{eq:discPortInt} are exact.
Furthermore, if $\beta_j\geq0$ for $j=1,\ldots,s$ (and this is the case for many collocation methods), we can deduce that $h\sum_{j=1}^s\dualp{e_d^j}{f_d^j}\leq0$, thus the dissipation component retains its qualitative behaviour.

\subsection{The quadratic Hamiltonian case}\label{ssec:collocation}

Let us consider the case where the Hamiltonian $\mc H(x)$ is a polynomial of degree (at most) 2, i.e.~it can be written as
\begin{equation*}
  \mc H(x) = \frac12x^TQx + v^Tx + c,
\end{equation*}
for some $Q=Q^T\in\R^{n,n}$, $v\in\R^n$ and $c\in\R$.
Since $\tilde x\in\R[t]_s^n$, we also have $H=\mc H\circ\tilde x\in\R[t]_{2s}$ and $\dot{H}\in\R[t]_{2s-1}$.
It is known that the maximum degree of exactness for quadrature rules with $s$ nodes is $2s-1$, and that it is attained only with Gaussian quadrature rules, that are associated with Gauss-Legendre collocation methods.
In particular, if we apply such a method, the integration of $\dot{H}$ will be exact, i.e.
\begin{equation*}
  \begin{split}
    H(t_f)-H(t_0)
    &= h\sum_{j=1}^s\beta_j\dualp{e_d^j}{f_d^j} + h\sum_{j=1}^s\beta_j\dualp{y_j}{u_j} = \\
    &= \int_{t_0}^{t_f}\big(\dualp{\tilde{e}_d(s)}{\tilde{f}_d(s)} + \dualp{\tilde y(s)}{\tilde u(s)}\big)\td s.
  \end{split}
\end{equation*}
In particular, since we always have $\beta_j\geq0$ for Gauss-Legendre collocation, the dissipation term is always non-positive, and we can write the discrete dissipation inequality
\begin{equation*}
  H(t_f)-H(t_0) \leq h\sum_{j=1}^s\beta_j\dualp{y_j}{u_j} = \int_{t_0}^{t_f}\dualp{\tilde y(s)}{\tilde u(s)}\td s.
\end{equation*}
Thus, in the quadratic Hamiltonian case the pH structure is preserved, in the sense that the PBE and the dissipation inequality have an exact discrete version.

\section{Examples}

\subsection{A basic DC power network example}

Let us consider as a toy example the linear electrical circuit represented in Fig.~\ref{fig:circuit},
where $R_G,R_L,R_R>0$ are resistances, $L>0$ is an inductor, $C_1,C_2>0$ are capacitors and $E_G$ is a controlled voltage source.
This can be interpreted as a basic representation of a DC generator ($E_G$,$R_G$), connected to a load ($R_R$) with a transmission line (the $\pi$ model given by $C_1,C_2,L,R_L$).
\begin{figure}[ht]
  \centering
  \begin{circuitikz}[scale=1.5,/tikz/circuitikz/bipoles/length=1cm]
    \draw (0,0) node[ground] {} to[american controlled voltage source,invert,v>=$E_G$] (0,1) to[R=$R_G$,i>=$I_G$] (0,2) to (1,2) to[L=$L$,i>=$I$] (2,2) to[R=$R_L$] (3,2) to (4,2) to[R=$R_R$,i<=$I_R$] (4,0) node[ground] {};
    \draw (1,2) to[C=$C_1$,v>=$V_1$,i<=$I_1$,*-] (1,0) node[ground] {};
    \draw (3,2) to[C,l_=$C_2$,v^>=$V_2$,i<_=$I_2$,*-] (3,0) node[ground] {};
  \end{circuitikz}
  \caption{Basic DC power network example}\label{fig:circuit}
\end{figure}
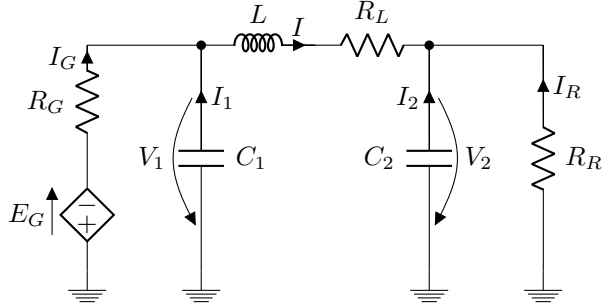
By means of Kirchhoff's circuit laws, this system can be naturally written as the following DAE:
\begin{equation}\label{eq:circuit}
  \begin{split}
    L\dot I &= -R_LI + V_2 - V_1, \\
    C_1\dot V_1 &= I - I_G, \\
    C_2\dot V_2 &= - I - I_R, \\
    0 &= -R_GI_G + V_1 + E_G, \\
    0 &= -R_RI_R + V_2.
  \end{split}
\end{equation}
The energy in the system can be stored in the inductor and in the two capacitors, giving the Hamiltonian
\begin{equation}
  \mc H(I,V_1,V_2) = \frac12LI^2 + \frac12C_1V_1^2 + \frac12C_2V_2^2.
\end{equation}
The system can then be written equivalently as an autonomous pHDAE of the form
\begin{subequations}\label{eq:circuit2}
  \begin{align}
    E\dot x &= (J-R)x + Bu, \\
    y &= B^Tx,
  \end{align}
\end{subequations}
with $x=(I,V_1,V_2,I_G,I_R)$, $E=\diag(L,C_1,C_2,0,0)$, $B=e_4$, $u=E_G$, $y=I_G$ and
\begin{equation*}
  J = \bmat{0 & -1 & 1 & 0 & 0 \\ 1 & 0 & 0 & -1 & 0 \\ -1 & 0 & 0 & 0 & -1 \\ 0 & 1 & 0 & 0 & 0 \\ 0 & 0 & 1 & 0 & 0}, \quad
  R = \bmat{R_L & 0 & 0 & 0 & 0 \\ 0 & 0 & 0 & 0 & 0 \\ 0 & 0 & 0 & 0 & 0 \\ 0 & 0 & 0 & R_G & 0 \\ 0 & 0 & 0 & 0 & R_R}.
\end{equation*}
The power balance equation reads as
\begin{equation*}
  \dot{\mc H} = -R_LI^2-R_GI_G^2-R_RI_R^2+I_GE_G;
\end{equation*}
in particular, if we shut down the generator ($E_G=0$), the Hamiltonian will decrease and converge to a solution such that $\dot{\mc H}=0$, that is, $I=I_G=I_R=0$.
The only state compatible with \eqref{eq:circuit} that satisfies that condition is $x=0$, so the system will always converge to that asymptotically stable point.

\subsection{Controlling the circuit}

Suppose now that $R_R$ represents a consumer, that requires a fixed amount of power $P=R_RI_R^2$ to be delivered to them.
We would like then to control the voltage of the generator $E_G$, so that the state of the system will converge to $I_R=I_R^*:=-\sqrt{P/R_R}$.
If we assume that a solution with $I_R\equiv I_R^*$ exists, then we also get $I=I_G\equiv-I_R^*$, $V_1\equiv(R_R+R_L)I_R^*$, $V_2\equiv R_RI_R^*$ and $E_G\equiv-(R_R+R_L+R_G)I_R^*$.

This can be interpreted in the following way, exploiting the port-Hamiltonian framework: let
\begin{equation*}
  x^* = \bmat{I^* \\ V_1^* \\ V_2^* \\ I_G^* \\ I_R^*} :=
  \sqrt{\frac{P}{R_R}}\bmat{1 \\ -R_R-R_L \\ -R_R \\ 1 \\ -1}
\end{equation*}
denote the desired state. By applying the change of variables $\tilde x=x-x^*$ to \eqref{eq:circuit2}, we get the equivalent pHDAE
\begin{subequations}\label{eq:circuit3}
  \begin{align}
    E\dot{\tilde x} &= (J-R)(\tilde x+x^*) + e_4u, \\
    y &= e_4^T(\tilde x+x^*),
  \end{align}
\end{subequations}
with the same Hamiltonian.
Since our goal is having $\tilde x\equiv0$ as an asymptotically stable solution, and $(J-R)x^*=-\sqrt{P/R_R}(R_R+R_L+R_G)e_4=:-e_4u^*$, by construction,
we write the equivalent system
\begin{subequations}\label{eq:circuit4}
  \begin{align}
    E\dot{\tilde x} &= (J-R)\tilde x + e_4\tilde u, \\
    \tilde y &= e_4^T\tilde x,
  \end{align}
\end{subequations}
with $\tilde u=u-u^*$ and $\tilde y=y-e_4^Tx^*=I_G-I_G^*$, which is again a pHDAE, but with Hamiltonian
\begin{equation}
  \tilde{\mc H}=\frac12L(I-I^*)^2+\frac12C_1(V_1-V_1^*)^2+\frac12C_2(V_2-V_2^*)^2.
\end{equation}
As before, this shows that choosing $\tilde u=0$ (i.e.~$u=-u^*$) would make the system converge to the desired state, that will be asymptotically stable.
On the other hand, this is not the only input that would satisfy this goal.
If we want to speed up the convergence, we have to increase the dissipation.
To do so, we can for example apply some linear feedback of the form $\tilde u=-\alpha\tilde y$ for some $\alpha>0$ (i.e.~$u=u^*-\alpha(I_G-I_G^*)$):
in this way, the dissipation inequality can be strengthened, since
\begin{align*}
  \dot{\tilde{\mc H}} &= -R_L\tilde I^2-(R_G+\alpha)\tilde I_G^2-R\tilde I_R^2 \leq \\
  &\leq -R_L\tilde I^2-R_G\tilde I_G^2-R\tilde I_R^2.
\end{align*}
From another point of view, if we apply a feedback of the form $u=u^*-\alpha(y-I_G^*)+\hat u$, then \eqref{eq:circuit4} can be written as
\begin{subequations}\label{eq:circuit5}
  \begin{align}
    E\dot{\tilde x} &= (J-R_\alpha)\tilde x + e_4\hat u, \\
    \tilde y &= e_4^T\tilde x,
  \end{align}
\end{subequations}
with $R_\alpha=R+\alpha e_4e_4^T>R$.

\subsection{Numerical simulations}

We present numerical simulations on the toy example \eqref{eq:circuit}.
We choose as constants $L=2$, $C_1=0.01$, $C_2=0.02$, $R_L=0.1$, $R_G=6$ and $R_R=3$.
This choice is not based on real world data, but it is made to reflect the usual assumption that, for transmission lines, $L\gg R_L\gg C$.
We discretize the system with respect to time, using the midpoint rule, which is the Gauss-Legendre collocation method with $s=1$ stages and order $p=2$ for ODEs.
Since \eqref{eq:circuit} is a semi-explicit DAE with index 1, the chosen method will also give convergence order 2 (see \cite[Theorem 5.16]{KunM06}).

First, we apply the time discretization to a system without control ($E_G=0$), starting from consistent non-zero initial values (see Fig.~\ref{fig:NCevo}).
The evolution of the state and of the Hamiltonian show the expected qualitative behaviour: after some time, they all converge to zero. Furthermore, while in the first half of the simulation the state oscillates, the Hamiltonian always decreases monotonically, following the discrete dissipation inequality.

\begin{figure}[ht]
  \centering
  \includegraphics[width=0.7\linewidth]{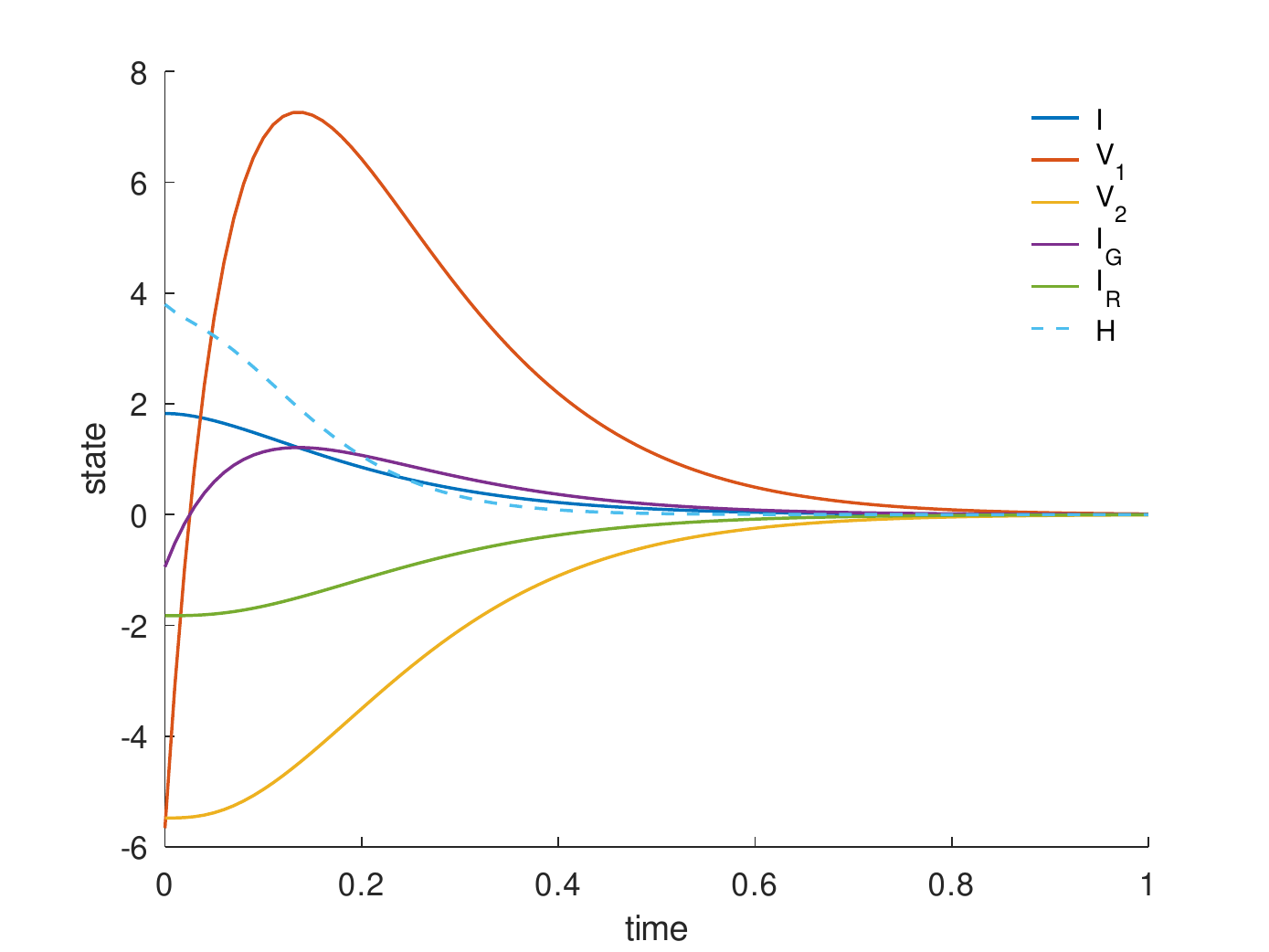}
  \caption{Evolution of state and energy with no control. The solid lines represent the state. The dashed line represents the Hamiltonian $\mc H$.}
  \label{fig:NCevo}
\end{figure}

As a second example, we apply control to the system, starting from almost-zero consistent initial values, with the goal of delivering $P=10$ power to the consumer.
This corresponds to the desired state
\begin{equation*}
  x^* = \bmat{I^* \\ V_1^* \\ V_2^* \\ I_G^* \\ I_R^*} :=
  \sqrt{\frac{10}{3}}\bmat{1 \\ -3.1 \\ -3 \\ 1 \\ -1}
  \approx \bset{
  \begin{array}{@{}r@{}}
    1.8257 \\ -5.6598 \\ -5.4772 \\ 1.8257 \\ -1.8257
  \end{array}
  },
\end{equation*}
and to the control value $u^*=9.1\sqrt{10/3}\approx 16.614$.
To simulate the fact that the change of input does not happen instantly, we actually choose as control
\begin{equation*}
  u(t) = u^*\Big(\arctan\big(5(t-0.5)\big)+0.5\Big),
\end{equation*}
so that the input increases smoothly and rapidly from 0 to $u^*$.
In Fig.~\ref{fig:Cevo} one can see that, after some initial oscillations, the state quickly converges to the desired values, represented by the dotted lines. At the same time, the second Hamiltonian $\tilde{\mc H}$ decreases monotonically, converging to zero, while the first Hamiltonian $\mc H$, representing actual energy in the system, converges to a positive value $\mc H^*$.

\begin{figure}[ht]
  \centering
  \includegraphics[width=0.7\linewidth]{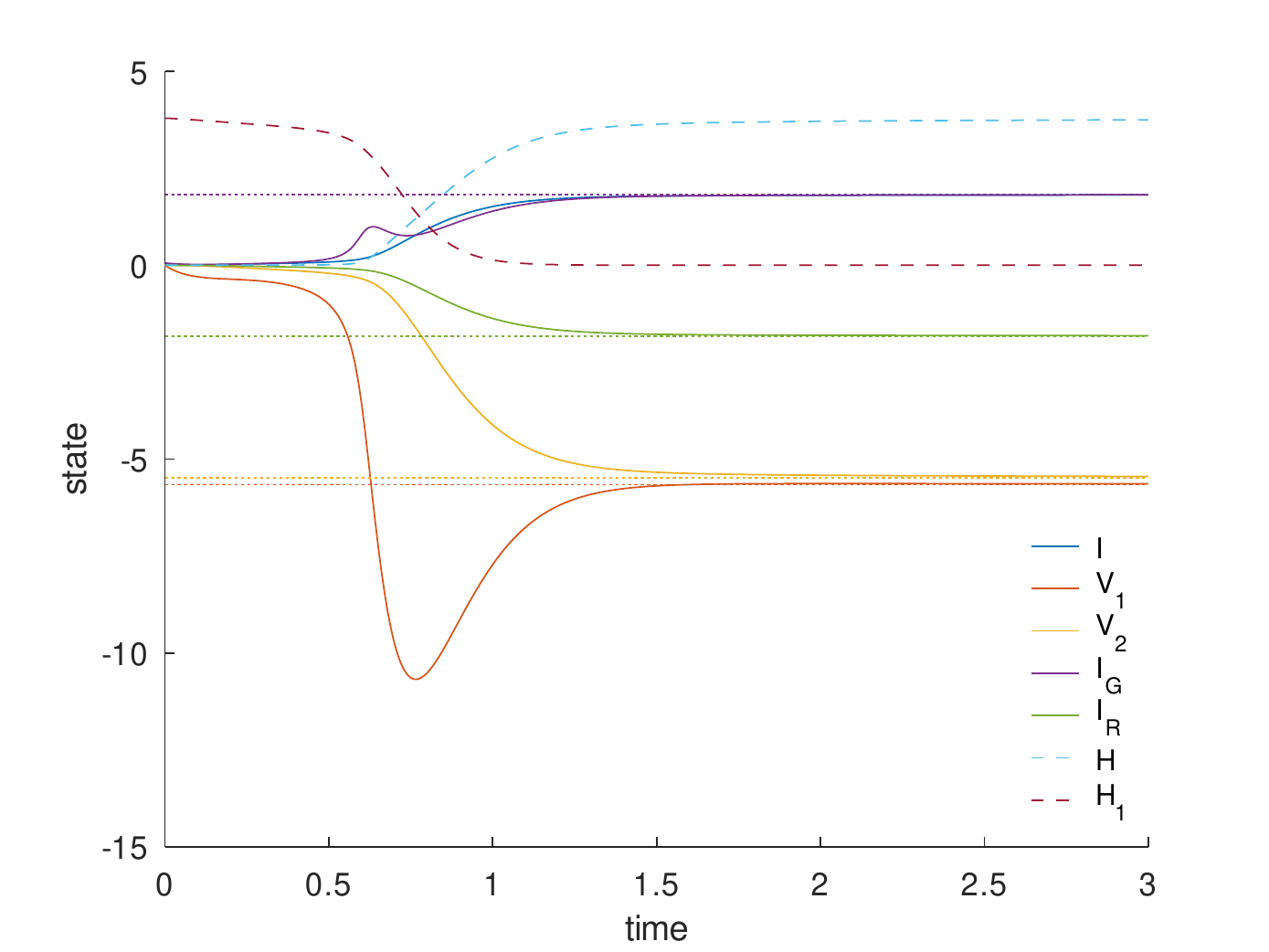}
  \caption{Evolution of state and energy with control. The solid lines represent the state of the system. The dotted lines represent the desired state. The cyan and purple dashed lines (H and H\textsubscript{1}) represent the first and second Hamiltonian $\mc H$ and $\tilde{\mc H}$, respectively.}
  \label{fig:Cevo}
\end{figure}

\section{Conclusions and future works}

\subsection{Conclusions}

We have presented a new definition for port-Hamiltonian descriptor systems, generalizing the one from \cite{BeaMXZ18} to include a larger class of equations. Extension to weak problems and to partial differential equations is also possible.
We verified that this definition satisfies several properties, in particular a dissipation inequality, invariance under a large class of variable transformations, and structure-preserving interconnection.
We generalized the definition of Dirac structure and we associated one to every system included in our formulation.
We extended the results from \cite{KotL18} to apply structure-preserving collocation schemes to port-Hamiltonian differential-algebraic equations.
Finally, we illustrated a simple example from the domain of electrical circuits that can be written with our formulation, and we presented some numerical experiments.

\subsection{Future Works}

Ongoing work is on the analysis of a larger class of numerical methods applied to pHDAEs, including more general Runge-Kutta schemes and partitioned methods, in particular in the case of semi-explicit DAEs with index 1.
Structure-preserving model reduction and space discretization for PDEs are also being considered.
Future work will include further analysis of the properties of pHDAEs, in particular the extension of the results from \cite{MehMW18} and possibly the characterization of controllability and observability of pH systems.

\section{Acknowledgments}

Riccardo Morandin is supported by the Deutsche Forschungsgemeinschaft (DFG) in the subproject B03, within the SFB Transregio 154.
Volker Mehrmann is supported by the German Federal Ministry of Education and Research (BMBF), in the project EiFer.



\bibliographystyle{plain}
\bibliography{MehM19}

\nocite{*}

\end{document}